\documentclass{article}

\usepackage{arxiv}

\usepackage{graphicx}      
\usepackage{natbib}        
\usepackage{amssymb,amsthm}
\usepackage{csquotes}
\usepackage[short]{optidef}
\usepackage{amsmath,amsthm}
\usepackage[pdfborder={0 0 0}]{hyperref}
\usepackage{url}            
\usepackage{booktabs}       
\usepackage{amsfonts}       
\usepackage{nicefrac}       
\usepackage{microtype}      
\usepackage{lipsum}		
\usepackage{graphicx}
\usepackage{natbib}
\usepackage{doi}
\usepackage{comment}

\newtheorem{theorem}{Theorem}
\newtheorem{assume}{Assumption}
\newtheorem{remark}{Remark}
\newtheorem{define}{Definition}

\title{ECCBO: An Inherently Safe Bayesian Optimization with Embedded Constraint Control  for Real-Time Optimization}

\author{ Dinesh Krishnamoorthy\thanks{Financial support from the NWO Veni Early Career Talent Scheme is gratefully acknowledged. } \\
	Department of Mechanical Engineering\\
	Eindhoven University of Technology\\
	5600 MB, Eindhoven \\
	\texttt{d.krishnamoorthy@tue.nl} \\
}

\date{}

\renewcommand{\headeright}{ }
\renewcommand{\undertitle}{}
\renewcommand{\shorttitle}{Multi-agent Bayesian Optimization}
\begin{document}

\maketitle

\begin{abstract}                

This paper introduces a model-free real-time optimization (RTO) framework based on unconstrained Bayesian optimization with embedded constraint control. The main contribution lies in demonstrating how this approach simplifies the black-box optimization problem while ensuring \enquote{always-feasible} setpoints, addressing a critical challenge in real-time optimization with unknown cost and constraints. Noting that controlling the constraint does not require detailed process models, the key idea of this paper is to control the constraints to \enquote{some} setpoint using simple feedback controllers. Bayesian optimization then computes the optimum setpoint for the constraint controllers. By searching over the  setpoints for the constraint controllers, as opposed to searching directly over the  RTO degrees of freedom,  this paper achieves an inherently safe and practical model-free RTO scheme. In particular, this paper shows that the proposed approach can achieve zero cumulative constraint violation without relying on assumptions about the Gaussian process model used in Bayesian optimization.The effectiveness of the proposed approach is demonstrated on a benchmark Williams-Otto reactor example.  
\end{abstract}



\section{Introduction}

Real-time optimization (RTO) plays a pivotal role in process operations due to its ability to  adapt to changing operating conditions. It ensures that processes can efficiently respond to changing factors like fluctuating demand or supply, thereby optimizing production processes in real-time. More importantly, it helps maintain  process safety and desired product qualities. 
The standard approach to real-time optimization involves  developing detailed process models that are updated using  measurements corresponding to steady-state process operation. The updated process models are then used within an optimization problem to compute the optimal setpoints \citep{chen1987TwoStepOpt,darby2011rto}. However, developing and maintaining detailed process models can be a major  hurdle that impedes practical implementation of RTO in industrial practice.

Model-free approach to online process optimization is an attractive alternative that circumvents the need for rigorous process models. There have been several attempts to develop model-free RTO schemes, most of which are based on online gradient estimation. The key underlying idea behind these approaches is to estimate the steady-state cost gradient directly from the process measurements, and drive it to zero using integral action, such that the necessary condition of optimality is asymptotically satisfied. See for example \cite{garcia1981optimal,golden1989adaptive,bamberger1978adaptive,francois2005NCO,dochain2011extremum}. However estimating the steady-state cost gradient  often leads to prohibitively slow  convergence, especially for process with long settling times \citep{srinivasan2011comparisonJuEst,DK2022RTOreview}.   Zeroth-order model-free optimization schemes such as Bayesian optimization is a promising alternative that avoids the need to estimate steady-state cost gradients.

Bayesian optimization is an effective method for optimizing unknown systems that finds the optimum by iteratively interacting with the system. BO uses a probabilistic surrogate model for the cost function, e.g., Gaussian process (GP), which is updated with each new observation. An acquisition function (also known as utility function)  is induced from the GP surrogate model that guides the search for the optimum by systematically trading-off exploration and exploitation. Unlike gradient-based methods,  Bayesian optimization calculates the next action independently of the current one, allowing for the possibility of significant jumps within the search space. While this notion of exploration speeds up the search for the global optimum, it can also lead to the exploration of infeasible setpoints that may violate safety-critical constraints at steady-state.

Constrained Bayesian optimization  is still an active area of research, especially for engineering applications \citep{berkenkamp2021bayesian,lu2022no,DK2022SafeBO,xu2022vabo}.  The classical approach to handling constraints is to scale the acquisition function with the probability of constraint feasibility  computed from the GP surrogates for the constraints  \citep{gelbert2012advanced}. Although such an approach would eventually converge to a feasible setpoint, it could potentially explore an infeasible setpoint. This is because the probability of constraint (in)feasibility at any point will be known with high probability only after observing the constraint at that point. As such, this approach does not have any guarantees on the cumulative violation of the constraints.  Another approach, known as Safe Bayesian optimization, ensures no constraint violation with high probability \citep{berkenkamp2021bayesian,DK2022SafeBO, DK2023SafeBO:RTO} under the assumption  that the constraint GPs are well calibrated (i.e., the true function is contained within the confidence intervals with high probability). Satisfying such an assumption in practice can often result in an overly cautious  algorithm, since the safe set may expand very slowly.  A third class of approach uses penalty functions, where the constraint violation is within  a prescribed budget, i.e. bounded cumulative constraint violation \citep{lu2022no,xu2022vabo}.

To this end, all the existing approaches to constrained BO requires probabilistic surrogate  models not just for the cost but also for all the unknown constraints. Safety guarantees and /or cumulative constraint violation bounds  only hold under the  assumption that the GP models for the constraints  are well calibrated. Simply put, calibrating the hyperparameters of  the constraint GPs can impact the constraint handling capability. 
Furthermore, the different approaches referenced above also assume that the constraints are independent, and uses individual GP surrogate models for each unknown constraint.

In the context of real-time process optimization, the information flow can be vertically decomposed, where the upper layer uses the setpoints to the lower layer as degrees of freedom, and the lower layer implements the actions to achieve the setpoint \citep{skogestad2023advanced}. 
Based on the fact that controlling the constraint to a setpoint does not require detailed process models, the key idea of this paper is to exploit decentralized feedback controllers (such as PID) to control the constraints to \enquote{some} setpoint, and use Bayesian optimization to find the optimal setpoints to the lower level  constraint controllers.  If a constraint  is optimally active, then the setpoint computed by Bayesian optimization will converge to the limiting value. Whereas, if a constraint is not optimally active, then the Bayesian optimization will converge to some feasible setpoint that optimizes the cost. By transforming the original optimization problem with embedded constraint controllers, BO now searches over the setpoint space, as opposed to searching directly over the RTO degrees of freedom. As such, the decision variables for the BO now only have simple box constraints, which can be trivially handled using standard  Bayesian optimization, and the box constraints inherently ensure that the RTO layer will never compute an infeasible setpoint.

To this end, this paper presents an inherently safe Bayesian optimization with constraint control embedded (ECCBO) that achieves zero cumulative constraint violation,  without relying on assumptions about the Gaussian process model used in Bayesian optimization. Since the focus of the paper is on the steady-state RTO layer, \enquote{safety} is defined w.r.t the constraints at steady-state operation.
The reminder of the paper is organized as follows: Section~2 recalls the preliminaries of RTO, constraint control, and Bayesian optimization. The proposed inherently safe Bayesian optimization algorithm with embedded constraint control (ECCBO) is presented in Section~3. Section~4 demonstrates the proposed approach using a benchmark Williams-Otto rector case study, before concluding the paper in Section~5.


%
\section{Preliminaries}

 \paragraph*{Real-time optimization and constraint control:}
 Consider the steady-state real-time optimization problem
 \begin{subequations}\label{Eq:RTO}
 \begin{align}
	\min_{\mathbf{u} \in \mathcal{U}}\; & J(\mathbf{u},\mathbf{d})\\
	\text{s.t.} \; & g_i(\mathbf{u},\mathbf{d}) \ge 0, i \in \mathbb{I}_{1:n}
\end{align}
 \end{subequations}
 where $ \mathbf{u} \in \mathcal{U} \subseteq \mathbb{R}^m $ denotes the set of decision variables for the RTO layer, $ \mathbf{d} \in \mathbb{R}^l $ denotes the set of disturbances, $ J: \mathcal{U} \times \mathbb{R}^l \rightarrow \mathbb{R}$ denotes the cost function, and $ g_i: \mathcal{U} \times \mathbb{R}^l \rightarrow \mathbb{R}$ for $ i = 1,\dots,n $ denotes the set of $ n $  constraints. The input constraint set $ \mathcal{U} = \mathcal{U}_{i} \times \dots \times \mathcal{U}_{m} $ is a simple set  composed of upper and lower bounds. Typically, the degrees of freedom $ \mathbf{u} $ for the RTO layer are provided as setpoints to the regulatory control layer below.

 For a process with $ n $ constraints, there can be utmost $ 2^n $ possible combinations of active constraints.
 The specific set of optimally active constraints, denoted as $ \mathbf{g}_{\mathbb{A}} \subseteq \mathbf{g}:= [{g_{1},\dots,g_{n}} ]^T$, may change depending on the operating conditions. If a constraint is known to be optimally active, indicated by $ g_{i} \subseteq \mathbf{g}_{\mathbb{A}} $ and $ g_{i}^* = 0 $, it can be controlled to a setpoint of its limiting value  (which is 0 when expressed in the positive null form as done in \eqref{Eq:RTO}) using   simple feedback controllers like PID, a technique known as active constraint control. However, for constraints that are not optimally active, denoted as $ g_{j} \nsubseteq \mathbf{g}_{\mathbb{A}} $, determining their non-trivial optimal values $ g_{j}(\mathbf{u}^*,\mathbf{d}) = z_{j}^* >0 $ becomes challenging.

 Typically, the RTO layer uses detailed process models to solve an optimization problem numerically to identify the  set of active constraints and the optimal values $ z_{j}^* >0 $ for the inactive constraints (assuming that the model used in the optimization solver captures the plant accurately).  However, this can be challenging if detailed process models are not available, which is the case that is considered in this paper.

\paragraph*{Bayesian Optimization:}
Bayesian optimization  aims to optimize  an unknown black-box function  \[ \min_{\mathbf{x} \in \mathcal{X}} F(\mathbf{x}) \] where $ \mathcal{X} $ is a simple set.  Bayesian optimization uses a probabilistic  model, such as Gaussian processes as a surrogate for the unknown function $ F(\mathbf{x}) $.
\begin{equation}\label{Eq:GP}
	F(\mathbf{x}) \sim \mathcal{GP}\left(\mu(\mathbf{x}),k(\mathbf{x},\mathbf{x}')\right)
\end{equation} where $ \mu(\mathbf{x})  := \mathbb{E}(F(\mathbf{x}))$ denotes the mean and $ k(\mathbf{x},\mathbf{x}') := \mathbb{E}\left[\left(F(\mathbf{x}) - \mu(\mathbf{x})\right)\left(F(\mathbf{x}') - \mu(\mathbf{x}')\right)\right] $  denotes the covariance of the unknown function.  Given noisy measurements $ y  $ of the unknown function $ F(\mathbf{x}) $, the Gaussian process model is updated by conditioning on the observations $ \{ (\mathbf{x}_{i}, y_{i})\}_{i} $ \citep{rasmussen2006gaussian}.
An acquisition function $ \alpha(\mathbf{x}) $ is induced from the Gaussian process surrogate, which  is then optimized to find the next action to take.
\begin{equation}\label{Eq:GPacq}
	\mathbf{x}_{t+1} = \arg \min_{\mathbf{x} \in \mathcal{X}} \alpha_t(\mathbf{x})
\end{equation}
where $ \alpha_t: \mathcal{X} \rightarrow \mathbb{R} $ is the acquisition function induced from the Gaussian process model conditioned on the past $ t $ measurements $ \{ (\mathbf{x}_{i}, y_{i})\}_{i=1}^t $. There are several choices for acquisition functions such as expected improvement, lower confidence bound, Thompson sampling etc. See for example \cite{frazier2018tutorial} and the references therein for further details. The crucial idea here is that the acquisition function balances exploration  and exploitation to find the optimum of the unknown function by sequentially interacting with the system.

\paragraph*{Contextual Bayesian optimization:}
To account for variables that are not directly under our control, such as prices, feed rates etc., the Bayesian optimization problem can be extended to account for these external factors. Here the optimization problem is formulates as
\[ \min_{\mathbf{x} \in \mathcal{X}} F(\mathbf{x},\mathbf{d}) \]
where $ \mathbf{d} \in \mathbb{R}^l$ denotes the external factors, also known as contexts. In this case, the GP model is used as a surrogate for $  F(\mathbf{x},\mathbf{d})  $ by conditioning on the observations $ \{ ([\mathbf{x}_{i},\mathbf{d}_{i}]^{\mathsf{T}}, y_{i})\}_{i} $. For a given contextual information $ \mathbf{d}_{t+1} $ at time step $ t+1 $, the next action is computed by optimizing
\begin{equation}\label{Eq:CGP}
	\mathbf{x}_{t+1} = \arg \min_{\mathbf{x} \in \mathcal{X}} \alpha_t(\mathbf{x},\mathbf{d}_{t+1})
\end{equation}
where $ \alpha_t: \mathcal{X} \times \mathbb{R}^l  \rightarrow \mathbb{R}$ is any suitable acquisition function derived from the GP surrogate conditioned on the past  $ t $ observations. The literature on constrained Bayesian optimization that accounts for additional nonlinear constraints is not included, as these will not be utilized in this paper.
\section{Bayesian Optimization with embedded constraint control (ECCBO)}
\begin{figure*}
	\centering
	\includegraphics[width=0.76\linewidth]{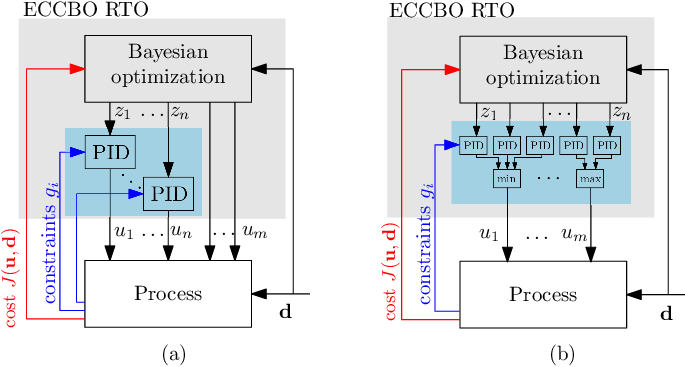}
	\caption{Proposed real-time optimization scheme using Bayesian optimization  highlighted in gray, with embedded constraint control (ECCBO) highlighted in blue. The process block shown here includes the lower level setpoint control layer and the plant.  (a) More degrees of freedom than constraint $ m>n $. (b) More constraints than degrees of freedom $ m<n $. Note that the process block  may further contain lower level regulatory controllers. }\label{Fig:ECCBO}
\end{figure*}
\begin{assume}\label{asm:blackbox}
	The cost and constraint functions $ J(\cdot,\cdot)  $ and $ g_{i}(\cdot,\cdot) $ in \eqref{Eq:RTO} are unknown, but are measured in real-time.
\end{assume}

The main objective is to find the optimum of the real-time optimization problem \eqref{Eq:RTO}, under Assumption~\ref{asm:blackbox}.   In search for the optimum using Bayesian optimization, we  want to ensure that no infeasible setpoints are explored.  


In this paper, we propose to exploit  constraint control  to safely find the optimum of the black-box system.
To do so, we transform the  inequality constrained problem \eqref{Eq:RTO} to the following equality constrained optimization problem.
 \begin{subequations}\label{Eq:RTO2}
	\begin{align}
		\min_{\mathbf{u},\mathbf{z}}\; & J(\mathbf{u},\mathbf{d})\\
		\text{s.t.}\;  & g_i(\mathbf{u},\mathbf{d}) = z_{i}, i \in \mathbb{I}_{1:n}
	\end{align}
\end{subequations}
where $ \mathbf{z} = [z_{1},\dots,z_{n}]^{\mathsf{T}} $ denotes the optimum value of the constraints.

We  use $ n $ degrees of freedom to tightly control the $ n $ constraints to their corresponding setpoints $ z_{i} $  using decentralized feedback control loops. The $ n $ degrees of freedom, and the $ n $ constraints are paired such that the steady-state relative gain array (RGA) of the resulting transfer matrix is non-negative and close to identity matrix at crossover frequencies. In order to minimize large time delays it would be preferable to pair constraints that are physically close to the manipulated variable (pair-close rule). One must also try to avoid pairing important constraints with  degree of freedom $ u_{i} $ with tight bounds $ \mathcal{U}_{i} $  \citep{DK2022RTOreview}.

We know that if a constraint $ g_{i}  $ is optimally active, then $ z_{i} = 0 $, and if a constraint $ g_{j} $ is not optimally active, then  we want $ z_{i} >0$. However, there are two challenges. Firstly, we do not know which of the constraints are optimally active, and which are not. Secondly, for the inactive constraints $ g_{j} \nsubseteq \mathbf{g}_{\mathbb{A}}$, we do not know  the exact value of $ z_{j}>0 $.
This paper proposes to use unconstrained Bayesian optimization to find the optimal setpoints $ z_{i}  \in [0,\bar{z}_{i}] $ for $ i =1,\dots,n $ constraint controllers, and for the remaining unconstrained degrees of freedom $ u_{n+1},\dots, u_{m} $.
\begin{assume}[Perfect control]\label{asm:perfectControl}
	The constraint controllers are designed such that the desired setpoint $ z_{i}  \in [0,\bar{z}_{i}]$ is asymptotically achieved for all $ i = 1,\dots,n $.
\end{assume}
Since the constraints are tightly controlled  to  their corresponding setpoints $ z_{i} $, we can equivalently rewrite \eqref{Eq:RTO2} as
	\begin{align} \label{Eq:RTO:BO}
	\min_{\mathbf{x} \in  \mathcal{X}}\;  F(\mathbf{x},\mathbf{d}) &=
		\min_{\mathbf{u}\in \mathcal{U},\mathbf{z}\in\mathcal{Z}}\;  J(\mathbf{u},\mathbf{d})\\
		& \qquad \text{s.t.}  \; g_i(\mathbf{u},\mathbf{d}) = z_{i}, i \in \mathbb{I}_{1:n} \nonumber
\end{align}
where $ \mathbf{x} := [z_{1},\dots, z_{n},u_{n+1},\dots,u_{m}]^T $.
Let us define $ \mathcal{Z}_{i} = [0,\bar{z}_{i}] $, where $ \bar{z}_{i} $ is some arbitrarily large upper bound for the constraint set points.  This leads to the constraint set  $ \mathcal{X} := \mathcal{Z}_{1} \times \dots \times \mathcal{Z}_{n}\times \mathcal{U}_{n+1}\times\dots \times \mathcal{U}_{m} $.

We can see that \eqref{Eq:RTO:BO} is an unconstrained black-box optimization problem, which can be solved using standard contextual Bayesian optimization. Here, we use a  Gaussian process as a surrogate for the cost function $ F(\mathbf{x},\mathbf{d})$ as a function of $\mathbf{x}$ and $ \mathbf{d} $. The GP is conditioned on the cost observed with the embedded constraint controllers that control $ g_{i} $ to a setpoint of $ z_{i} $ using  $ u_{i} $ as the manipulated variable.
\begin{remark}[Steady-state wait-time]
	Note that the Gaussian process model used in the Bayesian optimization algorithm is updated based on  steady-state cost measurement. As such, the BO-based RTO scheme requires a steady-state detection algorithm to ensure that GP models are conditioned on steady-state data. 
\end{remark}

Let us define the violation of the $ i^{th} $ constraint  at steady-state as $ v_{i,t} = -\left[g_{i}(\mathbf{u}_{t},\mathbf{d}_{t})\right]^- $ and the cumulative violation over $ T $ RTO iterations as \[  \mathcal{V}_{T} = \sum_{i=1}^{n}\sum_{t=1}^{T} v_{i,t}    \]
\begin{theorem}
	Under Assumption~\ref{asm:perfectControl}, the ECCBO framework achieves a cumulative violation of $ \mathcal{V}_{T} = 0 $ for any acquisition function.
\end{theorem}
\begin{proof}
	The Bayesian optimization using any acquisition function by design searches over the setpoints $ z_{i} \in [0,\bar{z}_{i}]$ for $ i = 1,\dots,n $. By controlling the constraints to a setpoint $ z_{i} \ge 0$ using the degrees of freedom $ u_{i} $, the constraint violation at steady-state $ v_{i,t} = 0 $ for all $ t\ge1 $, for all $ i = 1,\dots,n$, which results in $ \mathcal{V}_{T}  = 0$.
\end{proof}
Denoting $ F^* $ as the minimum of \eqref{Eq:RTO:BO} for a given disturbance $ \mathbf{d} $, querying an action $ \mathbf{x}_{t} $  at time $ t $ may lead to a gap at steady-state, which we define as the steady-state regret $ r_{t} = F(\mathbf{x}_{t},\mathbf{d})  - F^*$. The cumulative regret over $ T $ RTO runs is then defined as $ R_{T} =\sum_{t=0}^{T} r_{t}  $.
\begin{remark}[Regret]
	Since the Bayesian optimization in ECCBO reduces to a standard unconstrained Bayesian optimization, the regret bound for $  R_{T} $ follows from BO literature for specific acquisition functions under appropriate conditions \citep{chowdhury2017kernelized}.
	Bounding the contextual regret is in general  a challenging problem, as regret is evaluated in relation to the best action for each specific context. Consequently, the worst-case regret, e.g., when employing GP-LCB, to identify the optimal action $ \mathbf{x}^* $ can be determined after multiple instances of encountering a particular context. This is also reasonable in online process optimization of continuous processes, where the goal is optimize for the steady-state operating conditions, i.e. for a given $ \mathbf{d} $. As such, the standard regret results from the BO literature will also hold for  ECCBO.
\end{remark}

\begin{remark}[Overconstrained case $ m<n $]
 If the number of degrees of freedom is less than the number of constraints, we can have utmost $ m $ constraints optimally active.  In this case, max/min selector blocks can be used along with the $ n $ embedded constraint controllers. The choice of the selectors and the grouping of the constraint controllers for each selector block follows the selector design procedure as explained by \cite{DK2020selector}. The setpoint to the $ n $ controllers will be determined by the unconstrained Bayesian optimization. That is,  the decision variables for the transformed  Bayesian optimization problem are $ \mathbf{x} = [z_{1},\dots,z_{n}]^{\mathsf{T}}$.  The is schematically shown in Fig.~\ref{Fig:ECCBO}b.
\end{remark}
The design procedure for the proposed model-free RTO scheme, known as ECCBO-RTO can be summarized as follows:
\begin{itemize}
	\item Offline:
	 \begin{itemize}
\item[--]  Pair each constraint, denoted as $ g_{i} $, with an available degree of freedom using established pairing rules from process control literature \citep{DK2022RTOreview}.
\item[--]  Tightly tune the constraint controllers to minimize steady-state wait time, and at the same time ensuring clear time-scale separation from  any lower level  controllers.
\item[--]  Build a Gaussian process  to model the cost with respect to   the contexts $ \mathbf{d} $, the  setpoints $ z_{i} $ to the constraint controllers $ i=1,\dots,n $, and any remaining unconstrained degrees of freedom $ u_{{n+1}},\dots, u_{m} $.
 \end{itemize}
 	\item Online:
 \begin{itemize}
\item[--] If the process is operating at steady-state, update the Gaussian process model by conditioning on the past $ t $ observations of the cost measurement.
\item[--]  Induce an acquisition function $ \alpha_{t}(\mathbf{x},\mathbf{d}) $ based on the updated Gaussian process model, where $ \mathbf{x} := [z_{1},\dots,z_{n},u_{n+1},\dots,u_m]^{\mathsf{T}} $ contains the set of setpoints for the constraint controllers, as well as any remaining unconstrained degrees of freedom.
\item[--]  For a given context $ \mathbf{d}_{t+1} $, optimize the acquisition function $ \alpha(\mathbf{x},\mathbf{d}_{{t+1}}) $  to compute the actions $ \mathbf{x}_{{t+1}} $, which is implemented on the process.
  \end{itemize}
\end{itemize}

\section{Illustrative example}
This section demonstrates the proposed approach on a benchmark Williams-Otto reactor example, which converts raw materials A and B to useful products P and E, through a series of reactions
\begin{align*}
	A + B &\rightarrow C &k_{1} = 1.6599\times 10^{6}e^{-6666.7/T_{r}}\\
	B+C & \rightarrow P +E &k_{2} = 7.2177\times 10^{8}e^{-8333.3/T_{r}}\\
	C+P& \rightarrow G & k_{3} = 2.6745\times 10^{12}e^{-11111/T_{r}}
\end{align*}
The feedrate $ F_{A} $ with pure A component is an external disturbance. The feedrate $ F_{B} $ and the reactor temperature $ T_{r} $ are the two degrees of freedom, that are used to optimize the process.
\begin{align}\label{Eq:Cost}
	\max_{F_{B},T_{r}} \; & 1043.38x_{P}(F_{A} + F_{B})  + 20.92x_{E}(F_{A}+F_{B}) \nonumber \\
	&\quad - 79.23F_{A} - 118.34F_{B}  \\
	\textup{s.t.} & \quad x_{G} \leq 0.08 \nonumber \\
	& \quad x_{A} \leq 0.12 \nonumber
\end{align}
To design an ECCBO-RTO for this problem, we employ two SISO constraint control loops.
\begin{enumerate}
\item  Control the concentration of component $ G $ to a setpoint $z_{G} $ using reactor temperature $ T_{r} $.
\item Control the concentration of component $ A $ to a setpoint $z_{A} $ using feed rate $ F_{B} $.
\end{enumerate}

Contextual Bayesian optimization is then used to find the optimal setpoints $ z_{G} \in [0.07,0.08]$kg/kg and $ z_{A} \in [0.07,0.12]$kg/kg,  for the two control loops, respectively, for a given feed rate $ F_{A} $. Simply put, the decision variables for the Bayesian optimization are $ z_{G} $  and $ z_{A} $, and since the search space for this decision variables do not exceed 0.08 and 0.12, respectively, we can guarantee by design that the ECCBO-RTO layer will not compute any infeasible setpoints.
\begin{figure*}[t]
	\centering
	\includegraphics[width=0.99\linewidth]{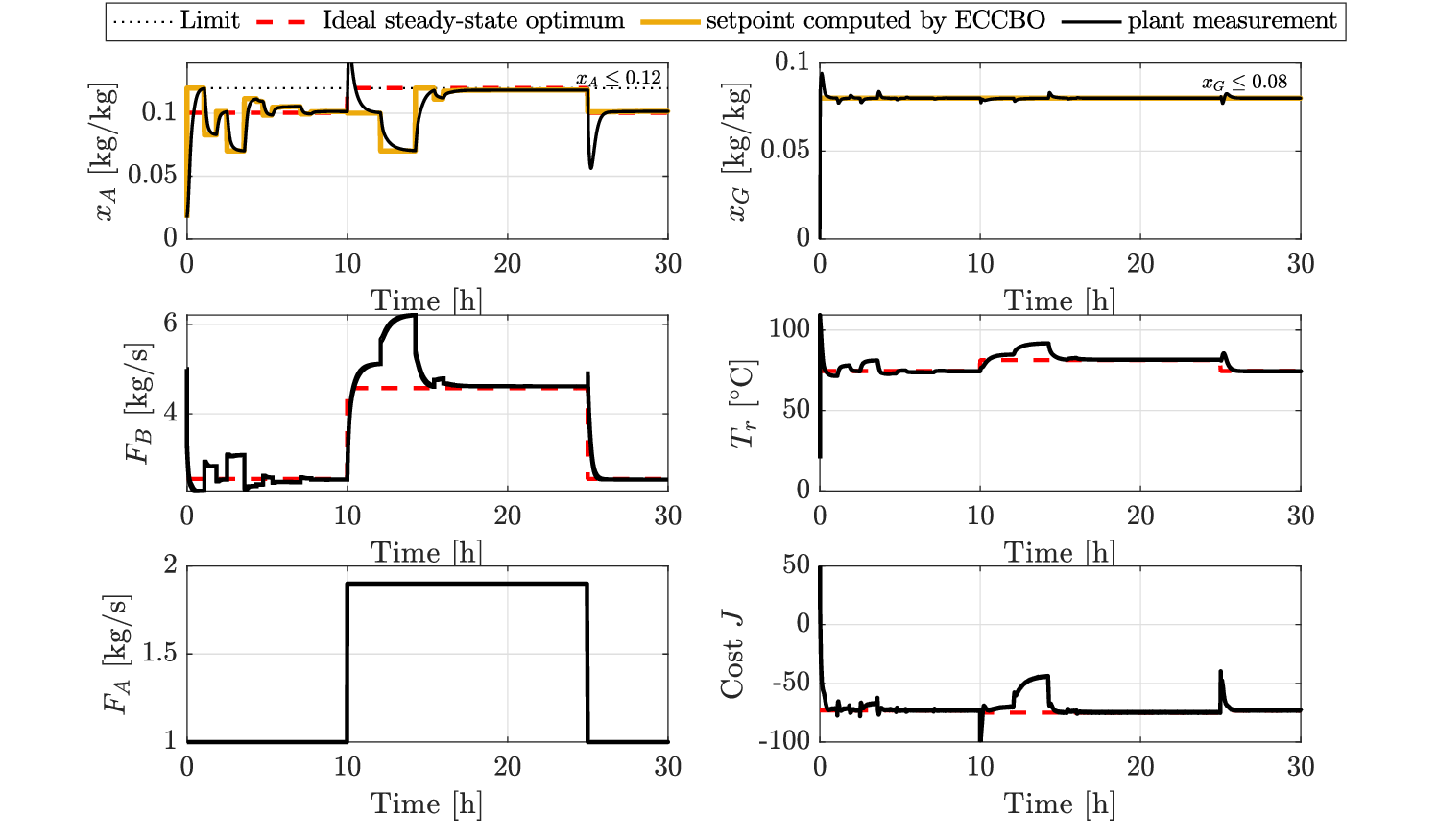}
	\caption{Simulation results shown the performance of our proposed approach (solid black), compared with the true steady-state optimum (red dashed lines). The decision variable for the Bayesian optimization, namely the setpoint for the constraint controller is shown in yellow solid lines. }\label{Fig:WOsim}
\end{figure*}
The plant measurements and the constraint controllers are sampled at a 1-second interval. For the two constraint control loops, the PI controllers are tuned using SIMC tuning rules.  The GP surrogate for the cost used by the Bayesian optimization algorithm uses the RBF kernel along with a bias kernel, and the Bayesian optimization uses the lower confidence bound  (GP-LCB) acquisition function\footnote{Traditional BO literature considers maximization problem, and hence uses the upper confidence bound (GP-UCB). The equivalent acquisition function for a minimization problem is the lower confidence bound. }. The Gaussian process is modeled using  \texttt{GPy} package \citep{gpy2014}, and the acquisition function is optimized using  \texttt{L-BFGS} algorithm from the \texttt{SciPy} library.

To initiate the Bayesian optimization routine, we employ a well-established steady-state detection algorithm, widely adopted in commercially available RTO tools. This algorithm is based on a statistical test that compares the total variance between consecutive signal points, ensuring the process operates at steady-state conditions as described by \cite{cao1995efficient,camara2016performance}.

The simulation initiates with a feed rate of $ F_{A} = 1 kg/s $, and changes to $ F_{A} = 1.9 $kg/s  at time $ t = 10$ h, and back to $ F_{A} = 1$kg/s  at time $ t = 25 $h. The set of active constraint changes, as the feed rate $ F_{A} $ varies. The true steady-state optimum for the different conditions are shown in red dashed lines in Fig.~\ref{Fig:WOsim}. Initially, when the feed rate $ F_{A} = 1  $kg/s, the constraint on $ x_{A} $ is not optimally active. Although this is unknown, the Bayesian optimization algorithm finds the optimum setpoint for $ x_{A} $ by sequentially interacting with the system. Since the search space for $ z_{A}  \in [0.07,0.12]$, we can see that the setpoints are always below 0.12, thus always ensuring constraint feasibility. The setpoint $ z_{A} $ computed by the BO algorithm is shown in yellow in the top left subplot in Fig. \ref{Fig:WOsim}.
For feedrate $ F_{A} = 1.9 $kg/s, the constraint on $x_{A}$ will be optimally active. In this case, the Bayesian optimization algorithm converges to the limiting value of $ z_{A} = 0.12 $kg/kg, without exploring any infeasible setpoints. Although at $ t=10 $h, the constraint is dynamically violated, the setpoint computed by ECCBO-RTO is feasible (i.e., no constraint violation at steady-state). Finally, when the feed rate reduces back to $ F_{A}  =1 $kg/s, the setpoint computed by the BO algorithm converges to the optimal setpoint in one step. This is because, the GP has previously seen the data for this \enquote{context}. Thanks to Gaussian processes conditioned on previously observed data, our algorithm converges to the optimum in one step.
For this process, the purity constraint is very low and from engineering insight we know that this constraint will likely be optimally active. This  can be incorporated in the GP prior, to avoid unnecessary explorations. This is reflected in the setpoint $ z_{G} $. 

In prior research, we assessed interior point-based safe Bayesian optimization and constrained Bayesian optimization on the same Williams-Otto reactor example. The findings illustrated in \cite[Fig.~5]{DK2023SafeBO:RTO} demonstrated that traditional constrained Bayesian optimization \citep{gelbert2012advanced} frequently explored infeasible setpoints, operating far beyond safe limits until the next RTO update. This limitation makes it impractical for real-time optimization in systems with safety-critical constraints. In contrast, safe Bayesian optimization, depicted in \cite[Fig.~4]{DK2023SafeBO:RTO}, avoided infeasible setpoints but required independent Gaussian process models for each constraint, which must be carefully calibrated. For the sake of brevity, these results are not reiterated in this paper, and readers are referred to \cite{DK2023SafeBO:RTO}.

\section{Conclusion}
In conclusion, this paper  introduced a novel model-free Real-Time Optimization (RTO) scheme, ECCBO (Bayesian optimization with Embedded Constraint Control). By leveraging the simplicity of constraint control, and the power of Bayesian optimization, this approach addresses a critical challenge in model-free real time optimization of process systems with unknown constraints. The main advantage of ECCBO is that it offers a practical solution that ensures feasibility and safety using well established control tools within the process control industry. In particular, we showed that using ECCBO we can get zero  cumulative constraint violation without imposing assumptions on the GP model used in the Bayesian optimization. Since ECCBO leverages the vertical decomposition with respect to the information flow in plant-wide control, it is important to note that this result is  not generalizable to  the broader constrained Bayesian optimization literature, where such lower level controllers are not implementable. 

\bibliographystyle{apalike}
\bibliography{BibRTO}             

\end{document}